\documentclass[10pt,oneside]{amsart}
\usepackage{amssymb, amscd, amsmath, amsthm, latexsym, hyperref}
\usepackage{pb-diagram, fancyhdr, graphicx, psfrag, enumerate}
\usepackage[text={6.3in,8.5in},centering,letterpaper,dvips]{geometry}
\newcommand{\compactlist}{\begin{list}{$\bullet$}{\setlength{\leftmargin}{1em}}}
\def\ba{{\boldsymbol \alpha}}
\def\zz{{\mathbb Z}}

\def\qq{{\mathbb Q}}
\def\cc{{\mathbb C}}
\def\ss{{\mathbb S}^1}

\def\cs{\mathbin{\#}}

\newcommand{\spinc}{\ifmmode{{\mathfrak s}}\else{${\mathfrak s}$\ }\fi}
\newcommand{\spinct}{\ifmmode{{\mathfrak t}}\else{${\mathfrak t}$\ }\fi}

\newcommand{\fig}[2] { \includegraphics[scale=#1]{#2} }

 \def\upj{\mathfrak{J}}
\def\downs{\underline{\mathfrak{S}}}
\def\ups{\overline{\mathfrak{S}}}

 \def\upjs{\mathfrak{j}}
\def\downss{\underline{\mathfrak{s}}}
\def\upss{\overline{\mathfrak{s}}}

\newtheorem{theorem}{Theorem}
\newtheorem{lemma}[theorem]{Lemma}
\newtheorem{corollary}[theorem]{Corollary}

\newtheorem{prop}[theorem]{Proposition}
\theoremstyle{definition}

\newtheorem{example}[theorem]{Example}

\newtheorem{innercustomthm}{Theorem}

{\endinnercustomthm}

\begin{document}
\title{Signature invariants related to the unknotting number}
\author{Charles Livingston}
\thanks{This work was supported by a grant from the National Science Foundation, NSF-DMS-1505586.}
\address{Charles Livingston: Department of Mathematics, Indiana University, Bloomington, IN 47405 }
\email{livingst@indiana.edu}


\begin{abstract}  New lower bounds on the unknotting number of a knot are constructed from the classical knot signature function.  These bounds can be   twice as strong as previously known signature bounds. They can also be stronger than known bounds arising from Heegaard Floer and Khovanov homology.   Results include new bounds on the Gordian distance between knots and information about   four-dimensional knot invariants.  By considering a related non-balanced signature function, bounds on the unknotting number of slice knots are constructed; these are related to the property of double-sliceness. \end{abstract}

\maketitle


\section{Introduction}

The unknotting number of a knot $K\subset S^3$, denoted $u(K)$, is the minimum number of crossing changes that is required to convert $K$ into an unknot.  This is among the  most intractable knot invariants.  For instance, the unknotting numbers of  several 10--crossing knots is still unknown.  Scharlemann~\cite{scharlemann} proved that the connected sum of two unknotting number one knots has unknotting number two, but little beyond this is known concerning the  additivity of the unknotting number.

Many knot  invariants offer tools for estimating the unknotting number; these include the rank of the homology of branched covers~\cite{kinoshita, wendt}, the    Murasugi signature~\cite{murasugi}, $\sigma(K)$, and the Levine-Tristram signature function~\cite{levine, tristram}, $\sigma_K(\omega)$, defined for $\omega \in S^1 \subset \cc$. Heegaard Floer homology and Khovanov homology have provided smooth knot invariants  $\tau$,   $\Upsilon$, and $s$,   see~\cite{os2,oss,rasmussen}, that also offer  lower bounds on the unknotting number. (See also,~\cite{cochran-lickorish, owens, owens1, owens-strle}.)  

The precise bound on the unknotting number that has been known to arise from the signature function is easily described.    Let $a = \max (\sigma_K(\omega))$ and $b = \min(\sigma_K(\omega))$.  Then   $u(K)  \ge    \left( a-b \right)/2  $.  Here we will observe that the knot signature function  offers much stronger constraints on the  unknotting number;   in some cases the new bounds will be seen to be twice as strong as this previously known  bound.  Examples also demonstrate that the new bounds can  exceed  those arising from Heegaard Floer and Khovanov homology.

There is a refined version of the unknotting number that incorporates the signs of the crossing changes that unknot $K$.  Let $U(K)$ be the set of integer pairs $(p,n)$ for which $K$ can be unknotted using $p$ crossing changes from positive to negative and $n$ crossing changes from negative to positive.  Then $U(K)$ is called the {\it signed unknotting set of $K$}.  Observe that $u(K) = \min\{p + n\  \big| \  (p,n) \in U(K)\}$.  Finding constraints  on $U(K)$ is especially difficult.  The results we present here depend critically on the signs of crossing changes, and thus they are able to extract information about $U(K)$ that cannot be attained with previously known techniques. In turn, these can be used to strengthen the bounds on $u(K)$.

The invariants we develop here also provide lower bounds on the  {\it Gordian distance} between knots $K$ and $J$, denoted $d_g(K,J)$; this is the minimum number of crossing changes that are required to convert $K$ into $J$.  Clearly $d_g(K,J) \le u(K) + u(J)$;  lower bounds are   more difficult to find.  

The results presented here also have applications to four-dimensional knot invariants.  For instance,  we provide new lower bounds on the {\it clasp number} of knots; this invariant is defined to be the  minimum number of transverse double points in an immersed disk in $B^4$ bounded by $K$; it is also referred to as the {\it four-ball crossing number} and is related to the notion of {\it kinkiness} defined by Gompf~\cite{gompf}. 

The signature function is built from a {\it non-balanced signature function}, $s_K(\omega)$.  The two functions agree almost everywhere, but $s_K$ is not a concordance invariant.  In a final section we discuss how $s_K$ provides bounds on the unknotting number that can be nontrivial for slice knots, and we present applications of this to double slicing of knots, a concept   dating to such work as~\cite{sumners, terasaka-hosakawa}. 
  \vskip.05in

\subsection{  Outline and summary of results} In Section~\ref{sec:defn}  we will review the definition of the signature function of a knot, $\sigma_K(\omega)$.  This is an integer-valued step function on the set of unit length complex numbers  $\omega \in  S^1 \subset \cc$; discontinuities can occur only at roots of the Alexander polynomial, $\Delta_K(t)$.  The definition of $\sigma_K$ is such that at each discontinuity  its value is equal to its two-sided average at that point.  There is also   a related jump function,  
$$J_K(e^{2 \pi i t }) = \frac{1}{2} \left( \lim_{\tau \to t^+ }\sigma_K(e^{2 \pi i \tau})   -\lim_{\tau \to t^-  }\sigma_K(e^{2 \pi i \tau} )\right).$$

The signature function is defined in terms of a Witt group; in Sections~\ref{sec:defn} and \ref{sec:change} we study this group and how crossing changes affect the Witt class associated to a knot. Section~\ref{sec:change} presents the proof of our key result.   In the statement of the theorem and throughout this paper, we denote the unit circle in the complex plane by $\ss$.

\begin{prop}\label{thm:changes} Let $K_+$ be a knot, let $\delta$ be an irreducible rational polynomial, and let $ \{\alpha_1, \ldots, \alpha_k\} \subset\ss$ with $k >0$ satisfy $\delta(\alpha_i) = 0 $ for all $i$.    If a crossing in a diagram  for $K_+$ is changed from positive to negative to yield a knot $K_-$, then one of the following two possibilities occurs.
\begin{enumerate}
\item For every $\alpha_i$,   $ J_{K_-} (\alpha_i)- J_{K_+}(\alpha_i)  = 0$ and $\sigma_{K_-}(\alpha_i) - \sigma_{K_+}(\alpha_i)\in \{0 , 2\}$.\vskip.05in
\item For every $\alpha_i$,   $ J_{K_-} (\alpha_i) - J_{K_+}(\alpha_i) \in \{-1, 1\}$ and $\sigma_{K_-}(\alpha_i) - \sigma_{K_+}(\alpha_i) = 1$.  
\end{enumerate}
\end{prop}

In Section~\ref{sec:alg} we prove a corollary to this proposition.

\begin{theorem}\label{thm:details} Let $K \subset  S^3$ be a knot, let $\delta(x)$ be a rational irreducible polynomial, and let  $ \{\alpha_1, \ldots, \alpha_k\} \subset \ss$ with $k >0$ satisfy $\delta(\alpha_i) = 0$ for all $i$.  Let $\upj_\delta$ denote the maximum   of $\{\big|J_K(\alpha_i)\big|\}$ and let $\downs_\delta$ and $\ups_\delta$ denote the minimum and maximum of $\{ \sigma_K(\alpha_i)\}$, respectively.    Suppose that  $\ups_\delta \ge 0$.  
\begin{enumerate} 
\item If $\downs_\delta \le \upj_\delta$, then $u(K) \ge \upj_\delta + ( \ups_\delta -\downs_\delta)/2$.\vskip.05in
\item  If $\downs_\delta \ge \upj_\delta$, then $u(K) \ge (\upj_\delta +  \ups_\delta)/2$. 
\end{enumerate}
\end{theorem}

\noindent{\bf Note.}  Letting $-K$ denote the mirror image of $K$, we have that $\sigma_{-K}(\omega) = -\sigma_K(\omega)$.    We also have that $u(-K) = u(K)$.  Thus, the  condition $\ups \ge 0$ does not limit the generality of Theorem~\ref{thm:details}.  The set of polynomials that are relevant in applying this theorem are symmetric factors of the Alexander polynomial of $K$, $\Delta_K(x)$.  The strongest obstructions arise by letting  $ \{\alpha_1, \ldots, \alpha_k\} $ be the full set of unit length roots of $\delta$.

Section~\ref{sec:alg} also presents an analog of Theorem~\ref{thm:details} in the case of signed unknotting numbers.  

\begin{theorem}\label{thm:signs} Let $K$,  $ \upj_\delta,  \downs_\delta$, and $ \ups_\delta$ be as in the statement of Theorem~\ref{thm:details}.  Suppose that $\ups_\delta \ge 0$. 
\begin{enumerate}

\item  If $\downs_\delta \le \upj_\delta$, then unknotting $K$ requires at least $(\upj_\delta + \ups_\delta)/2$ negative to positive crossings and $(\upj_\delta - \downs_\delta)/2$ positive to negative crossing changes.  \vskip.05in
\item If $\downs_\delta \ge \upj_\delta$, then unknotting $K$ requires at least $(\upj_\delta + \ups_\delta)/2$ negative to positive crossing changes. 
\end{enumerate}
\end{theorem}

In Section~\ref{sec:polynomial} we observe that the signed unknotting data obtained from different choices of polynomials can be complementary.  Using this, we provide examples for which combining the bounds that arise from different polynomials yields bounds on the (unsigned) unknotting number that are stronger than what can be obtained from either one of the polynomials. 
 
 In Section~\ref{sec:compare} we   construct explicit examples to demonstrate  that the bounds on the unknotting number provided by Theorem~\ref{thm:details} can be twice as strong as previously known signature bounds.     We also prove that our new bounds cannot exceed twice the classical bound.

Section~\ref{sec:gordian} discusses the application of these results to bounding the Gordian distance between knots.

In Section~\ref{sec:4d} we describe a four-dimensional perspective on these results. 
The obstructions we develop    actually  bound  the number of crossing changes required to convert $K$ into a knot with trivial signature function.  Thus, they  also bounds the number of crossing changes required to convert $K$ into a slice knot (the {\it slicing number} of $K$) and the number of crossing changes required to convert $K$ into an algebraically slice knot (the {\it algebraic slicing number}).  Past work on these invariants includes~\cite{livingston0, owens1, owens-strle}.  

In the remainder of Section~\ref{sec:4d} the focus is on the clasp number~\cite{murakami-yasuhara} of the knot $K$, which is the minimum number of   transverse double points in an immersed disk bounded by $K$ in the four-ball.  In the course of the work we also present a new simplified  proof of a result in~\cite{livingston1} that offers strong bounds on the  the cobordism distance between knots $K$ and $J$; this is the minimum genus of a cobordism $(W,F)$ between $(S^3, K)$ and $(S^3,J)$ with $W \cong S^3 \times I$.  References include~\cite{baader, baader1, feller, feller-krcatovich, hirasawa, kawamura,  owens1, owens-strle}.

In Section~\ref{sec:non-balanced} we briefly discuss the non-balanced signature function, $s_K(\omega)$, defined as the signature of the matrix denoted $W_F$ in
 Section~\ref{sec:defn}.  (The standard signature function, $\sigma_K(\omega)$, is built as the two-sided average of $\sigma_K(\omega)$.  The two functions agree almost everywhere, but $s_K(\omega)$ is not a concordance invariant.  As explained in the section, $s_K(\omega)$ provides bounds on the unknotting number of slice knots; from the four-dimensional perspective it is related to double-sliceness of knots.

\subsection{Example} To conclude this introduction, we provide a simple example illustrating  Theorem~\ref{thm:details}.

\vskip.05in 
\begin{example} \label{example1}
We prove the knot $  5_1 \cs 10_{132}$ has unknotting number 3.  To simplify our work, we let $K =  -5_1 \cs -10_{132}$ and prove $u(K) = 3$.  Working with the standard diagrams for $5_1$ and $10_{132}$, such as illustrated in~\cite{cha-livingston2, rolfsen}, one can quickly show that their unknotting numbers are at most $2$ and $1$, respectively, and thus $u(K) \le 3$.  We will prove that $u(K) = 3$ by showing $u(K) \ge 3$.    

The signature functions for $-5_1$ and $10_{132}$  and  the signature function for the difference, $K$, are illustrated in Figure~\ref{fig:graphs1}, graphed   as functions of $t$, where $\omega = e^{2 \pi i t}$, $0\le t \le 1/2$.    Let  $\delta$ be the tenth cyclotomic polynomial, $\phi_{10}$, having roots $\omega_1 = e^{2\pi i (1/10)} $ and  $\omega_2 = e^{2\pi i (3/10)}$ on the upper half circle.  As seen in the illustration, the jumps at $\omega_1$ and $\omega_2$ for $K$ are   0 and $2$, respectively.  The signatures are 0 and $2$ at these points.  

In the notation of Theorem~\ref{thm:details} we have $\upj_\delta = 2, \downs_\delta = 0$, and $\ups_\delta = 2$, and from that  theorem we have $u(K) \ge 2 + (2 - 0)/2 = 3$, as desired.
For this knot, the classical lower bound on the unknotting number that arises from the signature function is 2.  (The Rassmussen invariant $s$, the tau invariant $\tau$ and the Upsilon invariant, $\Upsilon$, all provide lower  bounds of 1.  For the first two, the values have been tabulated~\cite{cha-livingston2}.  Because $10_{132}$ is nonalternating, the computation of $\Upsilon_K$ is more complicated and will not be presented here.)

Applying Theorem~\ref{thm:signs}, we see that to unknot  $  5_1 \cs 10_{132}$ requires at least two crossing changes from positive to negative and one crossing change from negative to positive.
    \end{example}

\begin{figure}[h]
\fig{1}{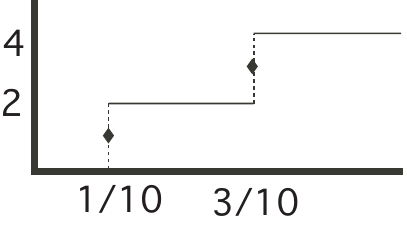} \hskip.2in \fig{ 1}{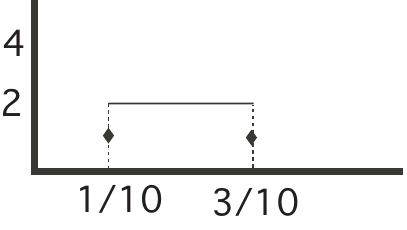}\hskip.2in \fig{.43}{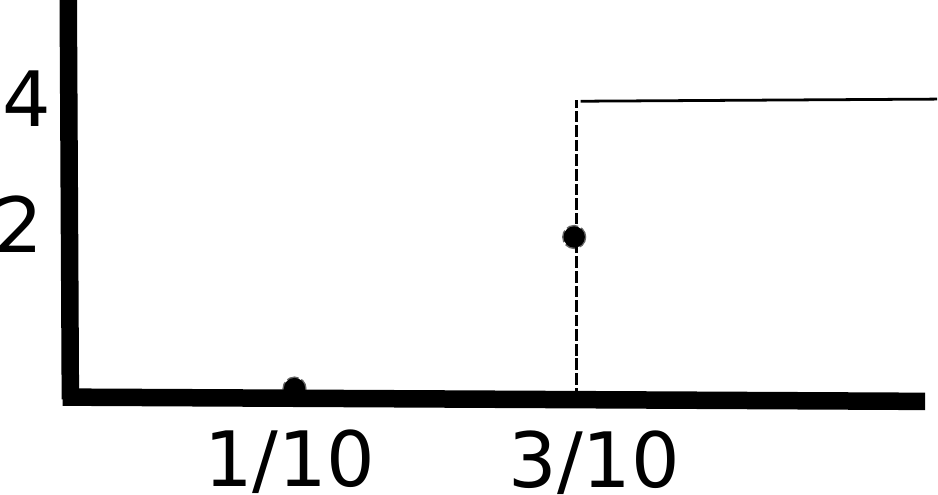} \caption{Signature functions for $-5_1,  10_{132}$ and $-5_1 \cs -10_{132}$.}\label{fig:graphs1}\end{figure}

\vskip.05in
\subsection{  Acknowledgments}  Many thanks to Pat Gilmer, whose careful reading    and suggestions helped eliminate several gaps and  greatly clarified the exposition.  Thanks are also due to Jeff Meier and Matthias Nagel for helpful comments.


\section{Witt class invariants of knots and the signature function} \label{sec:defn}

\subsection{The Witt class of a knot}
Let $F \subset S^3$ denote a  genus $g$ compact oriented surface with connected boundary $K$.  We will also write $F$ to denote  the surface  along with a choice of basis, calling this a {\it based surface}.  Associated to $F$ there is a $2g \times 2g$ Seifert matrix $V_F$.  Given $V_F$,  there is the matrix $W_F \in M_{2g,2g}(\qq(x))$ defined by 
$$W_F = (1-x)V_F + (1- x^{-1})V_F^{\sf T}.$$
Here $\qq(x)$ is the quotient field of $\qq[x, x^{-1}]$.  Elements of $\qq[x, x^{-1}]$ are Laurent polynomials; we will refer to elements of  $\qq[x,x^{-1}]$ simply as polynomials. 

For future reference, we recall that the Alexander polynomial of $K$ is given by $\Delta_K(x) = \det (V_F - xV_F^{\sf T}) \in \zz[x, x^{-1}]$ and note that $\det (W_F) = (1-x)^{2g}\Delta_K(x^{-1})$.  The Alexander polynomial is  well-defined up to multiplication by $\pm x^k$ for some $k$.

The Witt group $W(\qq(x))$ is defined to be the set of equivalence classes of nonsingular hermitian matrices with coefficients in $\qq(x)$, a field with involution $x \to x^{-1}$.  Two such matrices, $A$ and $B$, of ranks $m$ and $n$, are called  Witt equivalent if $m \equiv n \mod 2$ and  the form defined by $A \oplus - B$ vanishes on a subspace of dimension $(m+n)/2$.  The group structure on $W(\qq(x))$ is induced by direct sums and inversion is given by   multiplication by $-1$. Congruent matrices represent the same element in the Witt group.  For details concerning this Witt group, see~\cite{litherland}.  We have the following fundamental result of Levine~\cite{levine}.

\begin{prop} If $F_1$ and $F_2$ are   based Seifert surfaces for a knot $K$, then $W_{F_1} $ is Witt equivalent to $W_{F_2}$.
\end{prop} 

This permits us to define $W_K \in W(\qq(x))$ to be the Witt class represented by $W_F$ for an arbitrary choice of based Seifert surface $F$ for $K$.

\subsection{The signature function of a Witt class}
Suppose that  ${\sf w}  \in W(\qq(x))$ can be represented by two matrices, $A(x)$ and $B(x)$.  For almost all $\alpha \in \ss$, the matrices $A(\alpha)$ and $B(\alpha)$ are defined and nonsingular.  For all such $\alpha$, the signatures of $A(\alpha)$ and $B(\alpha)$, denoted $\sigma_A$ and $\sigma_B$, will be equal.  Thus for all real $t$,  there is an equality of limits: 
$$  \frac{1}{2}  \left( \lim_{\tau \to t^+ }\sigma_A(e^{2 \pi i \tau})    +\lim_{\tau \to t^-  }\sigma_A(e^{2 \pi i \tau} )\right)  =  \frac{1}{2} \left( \lim_{\tau \to t^+ }\sigma_B(e^{2 \pi i \tau})    +\lim_{\tau \to t^-  }\sigma_B(e^{2 \pi i \tau} )\right) .$$
For $\omega = e^{2 \pi i t}$, we denote this limit $\sigma_{\sf w}(\omega)$ and for  ${\sf w} = W_K$, we denote  it $\sigma_K(\omega)$.   This is a step function that is integer-valued except perhaps at its discontinuities, where it equals its two-sided average.  Modulo 2, its value (except at the discontinuities) equals the rank of a representative; thus, for a knot, it is even-valued away from the discontinuities and is integer-valued at the discontinuities. 
As in the introduction, for such a ${\sf w}$ we write 
$$J_{\sf w}(e^{2 \pi i t }) = \frac{1}{2} \left( \lim_{\tau \to t^+ }\sigma_{\sf w}(e^{2 \pi i \tau})   -\lim_{\tau \to t^-  }\sigma_{\sf w}(e^{2 \pi i \tau} )\right),$$ and in the case ${\sf w}= W_K$ we write $J_K(\omega)$.

Both of the  functions $\sigma_K$ and $J_K$ are invariant under complex conjugation.  They are defined on the set of unit complex numbers, which we henceforth write as $\ss = \{ \omega \in \cc\ \big|\ |\omega| = 1\}$.  A fairly simple exercise shows that for a knot $K$, the fact that  $\det(V_K-V_K^{\sf T}) = \pm 1$ implies that  $\sigma_K(\omega) = 0$ for all $\omega$ close to 1. Given the properties of $\sigma_K$,   when we graph $\sigma_K(e^{2 \pi t})$, we will restrict to $t \in (0,1/2)$. 

\subsection{The signature and the four-genus}
We now briefly summarize a well-known result  that follows immediately from~\cite{taylor}.

\begin{theorem}\label{thm:taylor2} If $K$ bounds a surface of genus $h$ in $B^4$, for instance if $g_4(K) = h$, then $W_K$ has a $2h$-dimensional representative.   \end{theorem}

\begin{proof}  According to~\cite{taylor}, if $K$ bounds a Seifert surface of genus $g$ and bounds a surface of genus $h \le g$ in $B^4$, then with respect to some basis, the upper left $(g - h)\times (g-h)$ block of the Seifert  matrix $V_F$ has all entries 0.  It then follows that $W_F$ is Witt equivalent to a sum $A \oplus B$, where $A$ is $2(g-h) \times 2(g-h)$ and is Witt trivial, and $B$ is $2h \times 2h$.

\end{proof}
There is the immediate corollary.
\begin{corollary} For all $t \in (0,2)$, $g_4(K)  \ge  \frac{1}{2}\left| \sigma_K(e^{2 \pi i t}) \right| $. \end{corollary}

\section{Diagonalization and Crossing changes}\label{sec:change}
\subsection{Diagonalizaton} The field $\qq(t)$ has characteristic 0, and thus the matrix $W_F$ associated to a Seifert surface $F$  is congruent to a diagonal matrix; that is, it  can be diagonalized using simultaneous row and column operations.  We will write a diagonalization by listing its diagonal elements:  $[d_1, d_2, \cdots, d_{2g}]$. By scaling the corresponding basis of the underling vector space, we can clear the denominators of these diagonal elements and divide  by factors   of the form $f(t) f(t^{-1})$.  Thus, we can assume that each $d_i$ is  a product of distinct irreducible symmetric polynomials in $\qq[t, t^{-1}]$.  

We   now have the  following.

\begin{theorem}\label{thm:parity}  Let ${\sf w} \in W(\qq(t))$, let $\delta $ be an irreducible symmetric polynomial, and let  $\ba = \{\alpha_i\} \subset \ss$ denote a subset of the roots of $\delta$ that lie on the unit circle.   Then if    $\sf w$  is represented by an $N \times N$ matrix, we have 
$$N  \ge  \Big( \max_{\alpha_i \in \ba}  \{ \left| J_{K}(\alpha_i) \right| \} +   \max_{\alpha_i \in \ba}   \{ \left| \sigma_{K}(\alpha_i) \right|  \} \Big). $$
For all $i$ and $j$, $ J_{K}(\alpha_i) = \sigma_{K}(\alpha_j)   \mod 2$. \end{theorem}
\begin{proof} 
Choose a  diagonal  representation of ${\sf w}$ of the form  $[f_1 \delta , \ldots, f_m \delta , g_1 , \ldots, g_n ]$, where the $f_i$ and $g_i$ are symmetric  polynomials that are relatively prime to $\delta$.     The jump at $\alpha_i$ is given as the sum of the jumps, each $\pm1$, arising from the diagonal elements of the form $f_i \delta$.  Thus,  $m \ge  \left| J(\alpha_i) \right| $ for all $i$.  The signature is determined by the signs of the $g_i$ at $\alpha_i$.  Thus, $n \ge  \left| \sigma_K(\alpha_i) \right| $ for all $i$.  This completes the proof of the inequality.

The prove  the last statement, concerning the parities of the jumps and signatures, we observe that modulo $2$, $J_K(\alpha_i) = m \mod 2$ and $\sigma_K(\alpha_j) = n \mod 2$.   In addition, $m + n = N$.  Finally, for a knot $K$, $W_F$ is a $2g \times 2g$ matrix.  The proof is completed by noting that  Witt equivalence preserves the rank of a representative, module two,

\end{proof}

\subsection{Crossing changes}
In   considering signed crossing changes, the following result is useful.  A proof can be constructed from a careful examination of Seifert's algorithm for constructing Seifert surfaces.  One proof is presented in~\cite{kim-livingston}, where the focus was on the  effect of crossing changes on the Alexander polynomial.

\begin{theorem}\label{thm:seifert} If $K_+$ and $K_-$ differ by a  crossing change from positive to negative, then they bound  Seifert surfaces $F_+$ and $F_-$ of the same genus, $g$, with the following property:  for   appropriate choices of bases for homology, the Seifert forms are identical except for the lower right entry:  $V_{F_-}^{2g,2g} - V_{F_+}^{2g,2g} = 1$.
\end{theorem}

\begin{lemma} \label{cor:block}Let $\delta$ be a symmetric irreducible polynomial.  The Witt classes for $K_+$ and $K_-$ decompose as
$$   W_{K_{\pm}} =  [f_1 \delta, \ldots , f_m\delta, g_1,\ldots ,  g_{n} ] \oplus 
\left(
\begin{array}{ c   c}
a(x) & \overline{b(x)} \\
{b(x)} &  d(x) +  \epsilon_{\pm} (1-x)(1-x^{-1})\\
\end{array}
\right),
$$
where the $f_i$ and $g_i$ are symmetric polynomials that are relatively prime to $\delta$,   $\epsilon_+ = 0$,  and  $\epsilon_- = 1 $. Furthermore, $m + n + 2 = 2g$. 
\end{lemma}
\begin{proof}   Consider the matrix representation of $W_{K_\pm}$ determined by the Seifert forms given in Theorem~\ref{thm:seifert}.  The determinant is nonzero:  an elementary exercise in linear algebra shows that the upper left  $(2g-1) \times (2g-1)$ submatrix has nullity at most one.  Thus, this block can be diagonalized via a change of basis so that the first $(2g-2)$ diagonal entries are nonzero.  The resulting $2g \times 2g$ matrix can have nonzero entries in the last column and bottom row, but the diagonal entries can be used to clear these out, with the possible  exception of the last two rows and columns.   This yields the desired decomposition.
\end{proof}
We can now prove Proposition~\ref{thm:changes}, which we restate.\vskip.05in

\noindent{\bf Proposition~\ref{thm:changes}.} 
{\it  Let $K_+$ be a knot, let $\delta$ be an irreducible rational polynomial, and let $ \{\alpha_1, \ldots, \alpha_k\} \subset\ss$ with $k >0$ satisfy $\delta(\alpha_i) = 0 $ for all $i$.    If a crossing in a diagram  for $K_+$ is changed from positive to negative to yield a knot $K_-$, then one of the following two possibilities occurs.

\begin{enumerate}
\item For every $\alpha_i$,   $ J_{K_-} (\alpha_i)- J_{K_+}(\alpha_i)  = 0$ and $\sigma_{K_-}(\alpha_i) - \sigma_{K_+}(\alpha_i) \in \{0, 2\}$.\vskip.05in
\item For every $\alpha_i$,   $ J_{K_-} (\alpha_i) - J_{K_+}(\alpha_i)  \in \{-1,  1\}$ and $\sigma_{K_-}(\alpha_i) - \sigma_{K_+}(\alpha_i) = 1$.
\end{enumerate} 
}
\vskip.05in

\begin{proof}  The difference $W_{K_-} - W_{K_+}$ is represented by the differences of the corresponding $2 \times 2$ block matrices given in Lemma~\ref{cor:block}, so we  restrict our attention to these, calling them ${\sf w}_-$ and ${\sf w}_+$.    
If the entry $a(x) = 0$, then ${\sf w}_+$ and ${\sf w}_-$ are  both Witt trivial, so the difference of jumps is 0, as is the difference of the signature; thus    Case (1) is satisfied.

If $a(x) \ne 0$, then the forms can be further diagonalized so that the only place at which they differ is the last diagonal element.  This diagonal element will be of the form 
$$ \frac{p(x)}{q(x)}  + \epsilon_{\pm} (1-x)(1- x^{-1}),$$ for some $p(x)$ and $q(x)$. We write the  $1 \times 1$ matrices as  $$v_+ =\bigg( \frac{p(x)}{q(x)}  \bigg)\text{\ \ \  and \ \ \ }   v_- =  \bigg(\frac{p(x)}{q(x)}  + (1-x)(1- x^{-1})\bigg).  $$  We will refer to the entries of these two matrices as $v_+(x)$ and $v_-(x)$.  It remains to analyze the jump functions and signature functions associated to these two matrices. 

For each value of $i$, we associate to $v_\pm$ the jump and signature of the form at $\alpha_i$, denoting these $j^i_\pm$ and $\sigma^i_\pm$.  We proceed in a series of steps.\vskip.05in

\noindent{\bf Step 1:} Consider $v_+$.  At points $\alpha$ close to but not equal to $\alpha_i$, the signature is either $1$ or $-1$.  If the signature changes sign at $\alpha_i$, then $\left| j^i_+ \right| = 1 $ and $\sigma^i_+ =0$.  On the other hand, if the signature doesn't change sign, then $\left| j^i_+ \right| = 0 $ and $\sigma^i_+ =\pm 1$.  The same properties hold for $v_-$.\vskip.05in

\noindent{\bf Step 2:}   Since $(1-\omega)(1- \omega^{-1}) >0$ for all $\omega  \in \ss$ with $\omega \ne 1$, we have that $\sigma^i_- - \sigma^i_+ \ge 0$. \vskip.05in

\noindent{\bf Step 3:}  Given Steps 1 and 2, the only possible nontrivial changes of the pairs $(\left| j^i_+  \right|, \sigma^i_+) \to (\left| j^i_- \right|, \sigma^i_-)$ are:
\begin{itemize} 

\item {\bf Type 1:} $(0,-1) \to (0, 1)$\vskip.05in
\item {\bf Type 2:} $(0,-1) \to (1, 0)$\vskip.05in
\item {\bf Type 3:} $(1,0) \to (0, 1)$\vskip.05in

\end{itemize}\vskip.05in
These are consisent with the statement of Proposition~\ref{thm:changes}.\vskip.05in

\noindent{\bf Step 4:}  The proof of the proposition is completed by showing that a nontrivial change of Type 2 or Type 3  occurs at some $\alpha_i$, the same change occurs at all $\alpha_i$.  After changes of basis, the forms can be written as Witt equivalent forms, for which we use the same names,
$$v_+ = (f_+(x)\delta(x)^{\epsilon_+}) \hskip.5in  v_+ = (f_-(x)\delta(x)^{\epsilon_-}) .$$
Here $f_\pm$ are symmetric polynomials that are relatively prime to $\delta$ and $\epsilon_\pm$ are either 0 or 1.  There are four cases to consider.
\begin{itemize}
\item If $(\epsilon_+, \epsilon_-) = (0,0)$, then there are not nontrivial jumps at any $\alpha_i$, so no changes of Type 2 or 3 occur.\vskip.05in
\item If $(\epsilon_+, \epsilon_-) = (1,0)$, then at each $\alpha_i$ there is jump for $v_+$ but not for $v_-$, so for all $\alpha_i$ we see a change of Type 3.
\item If $(\epsilon_+, \epsilon_-) = (0,1)$, then at each $\alpha_i$ there is no $v_+$ jump but for  $v_-$ there is  a jump, so for all $\alpha_i$ we see a change of Type 2.\vskip.05in
 \item If $(\epsilon_+, \epsilon_-) = (1,1)$, then at each $\alpha_i$ both  $v_+$and  $v_-$ have nonzero jumps, so     no change of Type 2 or 3 occurs at any of the  $\alpha_i$. \vskip.05in

\end{itemize}
Together, these   steps complete the proof of the proposition.
 \end{proof}


\section{Bounds on the unknotting number and signed unknotting number}\label{sec:alg}
 
 To begin, we have a corollary of Proposition~\ref{thm:changes}.
\vskip.05in
\begin{corollary} {\it  Let $K \subset  S^3$ be a knot and let  $ \{\alpha_1, \ldots, \alpha_k\} \subset \ss$  be a nonempty subset of the complex roots of an irreducible rational polynomial.  Let $\upj$ denote the maximum   of $\{\big|J_K(\alpha_i)\big|\}$;   let $\downs$ and $\ups$ denote the minimum and maximum of $\{ \sigma_K(\alpha_i)\}$.    A crossing change in $K$ from positive to negative either: (1) leaves $\upj$ unchanged and leaves each of  $\downs$ and $\ups$   unchanged or increased by 2; or (2) changes $\upj$ by 1 and increases both $\downs$ and $\ups $ by 1.
}
\end{corollary}
\vskip.05in

\begin{proof}Changes of the first type in Proposition~\ref{thm:changes} clearly leave $\upj$ unchanged and leave each of $\downs$ and $\ups$ unchanged or increased by 2. 

Changes of the second type, since they increase every signature by $1$, increase the minimum and maximum signature  by 1.  The condition on the change in $\upj$ is a little more subtle.  For instance, if it were possible that one jump is $1$ and one jump is $2$, then after the change, it might be that the first jump is 2 and the second is 1, and thus the maximum absolute value would not change.  However, as stated in Theorem~\ref{thm:parity},  the jumps all have the same parity.  Thus, the parity of the jumps switch for such a crossing change, so the maximum absolute value must also change.

\end{proof}

\subsection{Unsigned unknotting number bounds}
Our main goal in this section is the following theorem, as stated in the introduction.  \vskip.05in
\noindent{{\bf Theorem \ref{thm:details}.}  {\it  Let $K \subset  S^3$ be a knot and let  $ \{\alpha_1, \ldots, \alpha_k\} \subset \ss$  be a nonempty subset of the complex roots of an irreducible rational polynomial.  Let $\upj$ denote the maximum   of $\{\big|J_K(\alpha_i)\big|\}$;   let $\downs$ and $\ups$ denote the minimum and maximum of $\{ \sigma_K(\alpha_i)\}$.    Suppose that  $\ups \ge 0$.  If $\downs \le \upj$ then $u(K) \ge \upj + ( \ups -\downs)/2$.  If $\downs \ge \upj$ then $u(K) \ge (\upj +  \ups)/2$. }
\vskip.05in

\begin{proof} We consider the set 
$$\Lambda =  \left\{    \left( \upjs,    \downss , \upss    \right)  \in   \zz \oplus \zz \oplus\zz \  \big|\   \upjs = \downss = \upss \mod 2    \right\}.$$ 
 We define two sets of functions from $\Lambda$ to itself.  The first  set consists of what we call  {\it $F$--type functions}.  These, which do not change the value of  $\upjs$, are as follows:
\begin{itemize}
\item $F^-_1( \upjs, \downss , \upss)  =  (\upjs, \downss -2 , \upss)$\vskip.05in
\item
$F^-_2(  \upjs, \downss , \upss)  =  (\upjs,  \downss , \upss-2)$\vskip.05in
\item $F^-_3(  \upjs, \downss , \upss)  =  (\upjs,  \downss -2 , \upss-2)$\vskip.05in
\item $F^+_1(  \upjs, \downss , \upss)  =  (\upjs,  \downss +2 , \upss)$\vskip.05in
\item
$F^+_2(  \upjs, \downss , \upss)  =  (\upjs,  \downss, \upss+2)$\vskip.05in
\item $F^+_3(  \upjs, \downss , \upss)  =  (\upjs,  \downss +2 , \upss+2)$.
\end{itemize}
Functions of the second type, {\it $G$--type functions}, change the value of $\upjs$. 
These are defined as follows:
\begin{itemize}
\item $G^-_1( \upjs, \downss , \upss)  = (\upjs -1,  \downss -1 , \upss -1)$\vskip.05in

\item $G^-_2(  \upjs, \downss , \upss)  =  (\upjs +1,  \downss -1 , \upss -1)$\vskip.05in

\item $G^+_1(  \upjs, \downss , \upss)  = (\upjs - 1,  \downss +1 , \upss +1)$\vskip.05in

\item $G^+_2(  \upjs, \downss , \upss)  = (\upjs + 1,  \downss +1 , \upss +1)$.\vskip.05in

\end{itemize}

For a given knot $K$, a crossing change affects the value of the associated pair $(\upj, \downs, \ups)$ by applying one of these functions.   The superscripts $+$ and $-$ correspond to whether the crossing changes is positive to negative or negative to positive, respectively.  A sequence of crossing changes that results in an unknot yields a sequence of   these functions which in composition carry  $(\upj, \downs, \ups)$ to $(0,0,0)$.  

We now consider a given element $(\upj,  \downs , \ups)\in \Lambda$.  For the proof of the theorem, we can assume $\upj \ge 0$ and $\downs \le \ups$.   We ask for the minimum length of a sequence of these functions that can reduce  $(\upj,  \downs , \ups) $ to $(0, 0,0)$.  A simple observation is that the  $F$--type functions commute with the $G$--type functions, so we can assume that a minimal length sequence consists of a sequence of $G$--type functions followed by a sequence of $F$--type functions.  (Here, the order of the sequence is in terms of the order of composition; in function notation, $f \circ g$ denotes $g$ followed by $f$.)

Since the $F$--type functions do not change the value of $\upjs$,  the initial application of the $G$-type functions reduces $\upj$ to $0$.   It follows that by commuting elements in the initial sequence of $G$--type functions, we can assume the sequence begins with $\upj$ terms of type $G_1^\pm$ (which together decrease the  $\upjs$--coordinate to 0) followed by a  sequence of $G$--type functions that alternately increase and decrease the $\upjs$--coordinate by 1. 

Next observe that a pair of $G$--functions that raise and then lower the $\upjs$--coordinate compose to give a single function, either the identity or one of  $F_3^-$ or $F_3^+$. Thus, in a minimum length sequence, such pairs do not appear, and hence there are   precisely $\upj$ of the $G$--type functions followed by a sequence of $F$--type functions.

If one considers all possible sequences of  $G$--type functions of length $\upj$ that convert $(\upj,\downs,\ups)$ to a triple with $\upjs$--coordinate 0, the possible ending values of $(\downss, \upss)$ are $(\downs +\alpha, \ups +\alpha)$, where $-\upj \le \alpha \le \upj$.   Each $F$--type function reduces the difference $\ups - \downs$ by at most 2.  Thus at least 
$$\left( (\ups +\alpha) - (\downs +\alpha) \right)/2 =  (\ups - \downs)/2$$ applications of $F$--type functions are required to reduce this pair to $(0,0)$.  In fact, if for some $\alpha$ the interval  
$(\downs +\alpha, \ups +\alpha)$ contains 0, a sequence of that length will suffice.  There will be such an $\alpha$ if $\downs \le \upj$.  Thus, in this setting  the minimum length sequence is $\upj +(\ups - \downs)/2$, as desired.

 On the other hand, if $\downs > \upj$, then we also have $\ups >\upj$.  In this case, the sequence of $G$--type functions has reduced the $\upss$--coordinate to no less than $\ups - \upj$, so at least another $(\ups - \upj)/2$ steps are required.  Thus, the minimal length of the sequence is at least $$\upj + (\downs - \upj)/2 + (\ups - \downs)/2 = (\upj + \ups)/2.$$  This completes the proof of Theorem~\ref{thm:details}.
\end{proof}

 \subsection{Signed unknotting number bounds}

 In the proof of Theorem~\ref{thm:details}, at one step we considered  the condition that an interval $[ \downs +\alpha, \ups +\alpha ]$ contained 0.  If the argument is examined closely, in the case that $\downs < \upj$  there can be more than one  $\alpha$ for which this holds.  The effect of this is to complicate the count of negative and positive  shifts that will appear in the sequence of functions that reduce the jumps and signatures to 0.  \vskip.05in
 
\noindent{\bf Theorem~\ref{thm:signs}.}  {\it   Let $K$  and $(\upj, \downs , \ups )$ be as in the statement of Theorem~\ref{thm:details}.  Suppose that $\ups \ge 0$. 
\begin{enumerate} 

\item  If $\downs \le \upj$, then unknotting $K$ requires at least $(\upj + \ups)/2$ negative to positive crossings and $(\upj - \downs)/2$ positive to negative crossing changes.  \vskip.05in
\item If $\downs \ge \upj$, then unknotting $K$ requires at least $(\upj + \ups)/2$ negative to positive crossing changes. 
\end{enumerate}}

\begin{proof} Suppose that the sequence of functions that reduces $\upj$ to $0$ has $a$ terms that lower the $\downss$ and $\upss$ coordinates. That sequence has $\upj - a$ terms that increase the  $\downss$ and $\upss$ coordinates.  The application of these functions carries the pair $(\downs, \ups)$ to $(\downs +\upj -2a, \ups +\upj -2a)$.  Assume this interval contains 0.  Then the sequence of $F$--type functions that carry this pair to $(0,0)$  must have $-( \downs +\upj -2a)/2$ terms that increase the smaller coordinate and  $( \ups +\upj -2a)/2$ terms that decrease the larger coordinate.  Summing these counts gives the desired result.
\end{proof}

\begin{example}  From Example~\ref{example1} we see that for the knots $-5_1 \cs -10_{132}$ and $-5_1, \cs 10_{132}$   $(\upj,  \downs, \ups )$ is  
 $ (2,  0, 2)$ or $(2, 2,4)$, respectively.  In both cases $\downs \le \upj$. Thus, applying Theorem~\ref{thm:signs}, we see that unknotting $-5_1 \cs -10_{132}$ requires at least two crossing changes from negative to positive and one crossing change from positive to negative.  To unknot $-5_1  \cs 10_{132}$ requires at least three crossing changes from negative to positive.
\end{example}

\section{Polynomial splittings and signed unknotting numbers}\label{sec:polynomial}

The bounds on the unknotting number developed in the previous sections depend on a choice of polynomial.  This section presents an example for which there are two relevant polynomials to consider.  Either one provides a lower bound of three for the unknotting number.  However, for one of the polynomials,  when signs are considered it will be seen that unknotting requires at least two changes from negative to positive and one change from positive to negative.  Using the other polynomial, we will see that at least three changes from negative to positive are required.  Combining these two results, we see that at least three changes from negative to positive and one change from positive to negative are required, and hence the unknotting number must be at least four.  
 
\begin{example} \label{ex:split}

We consider the knot $K = 2(3_1) -5_1- 8_2 +10_{132} -11_{n6}$. Figure~\ref{fig:twopoly} illustrates the graph of its signature function. The scale is such that $\sigma_K(\alpha_1) = 2$.  The data we use, including the signature function, can be found in~\cite{cha-livingston}.

\begin{figure}[h]
\fig{.8}{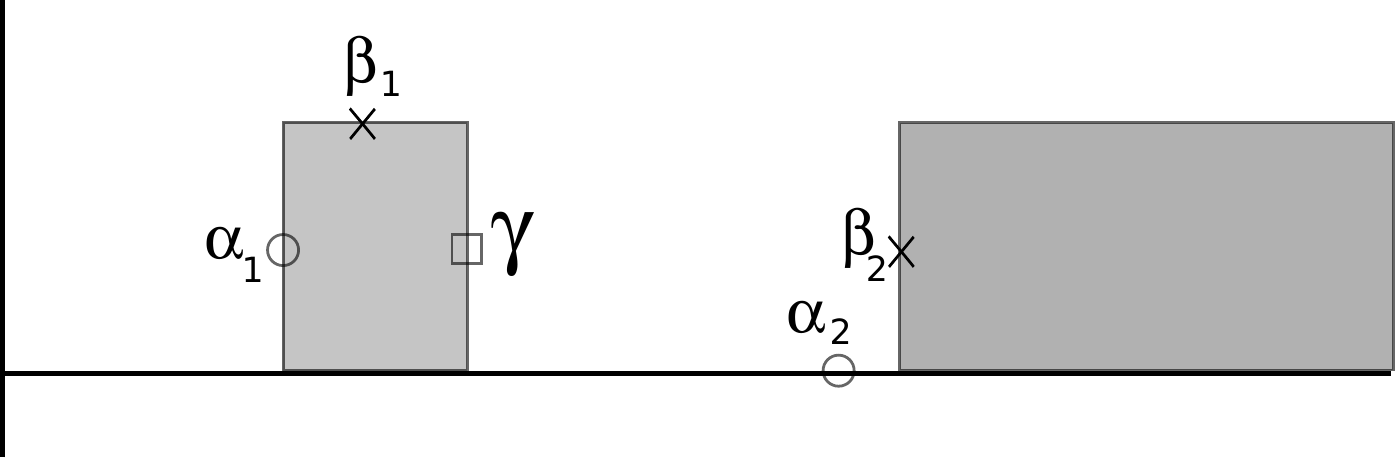}\label{fig:twopoly}
 \caption{Signature function for $2(3_1) -5_1- 8_2 +10_{132} -11_{n6}$.}\label{fig:twopolygraph}\end{figure}
 
 The relevant roots of the Alexander polynomial are of the form $e^{2\pi i t}$ with $0<t < 1/2$. 
 \begin{itemize}
 \item $K = 3_1$: \vskip.05in
 \begin{enumerate}[{  $\bullet$}]
 \item  $\delta_1 =\Delta_K(x) = x^2 - x +1$
 \vskip.05in \item roots ($\gamma$), $t \approx .167$.\vskip.05in
 \end{enumerate}

 \item $K = 5_1$ or $K=10_{132}$: \vskip.05in
 \begin{enumerate}[{  $\bullet$}]
 \item  $ \delta_2 = \Delta_K(x) =x^4 - x^3 + x^2 - x +1$\vskip.05in
  \item roots ($\alpha_1, \alpha_2$), $t = .1, t=.3$.
 \end{enumerate}\vskip.05in
 
 \item $K = 8_2$ or $K=11_{n6}$\vskip.05in
  \begin{enumerate} [{  $\bullet$}]
  \item  $ \delta_3 = \Delta_K(x) =x^6 -3x^5+3x^4-3x^3 +3x^2-3x +1$\vskip.05in
   \item roots ($\beta_1, \beta_2$), $t \approx  .132, t\approx .322$.
  \end{enumerate}\vskip.05in
 
 \end{itemize}

 The jump and signature data is as follows:
 \begin{itemize}
 \item $\upj_{\delta_1}(K) = 2$, $\downs_{\delta_1}(K) = 2$, $\ups_{\delta_1}(K) = 2$\vskip.05in
 
 \item $\upj_{\delta_2}(K) = 2$, $\downs_{\delta_2}(K) = 0$, $\ups_{\delta_2}(K) = 2$\vskip.05in
 
 \item $\upj_{\delta_3}(K) = 2$, $\downs_{\delta_3}(K) = 2$, $\ups_{\delta_3}(K) =4$\vskip.05in
 
 \end{itemize}
 
 From the $\delta_2$ invariants and the $\delta_3$ invariants we see that at least three crossing changes are required to unknot $K$.  However, from the $\delta_2$ invariants we see that an unknotting requires at least two negative to positive changes and at least one positive to negative change are required.  From $\delta_3$ we see that at least three negative to positive changes are required.  Combining these observations, we see that at least three negative to positive changes are required, and at least one positive to negative change is needed.  Thus, the unknotting number is at least four.\vskip.05in
 
 \noindent{\bf Note.} It is evident and can be proved in a number of ways that the unknotting number of this knot is much greater than four.  This example becomes more interesting when considered from the four-dimensional perspective, as discussed in Section~\ref{sec:4d}.  It follows from the results there that this knot does not bound an immersed disk in $B^4$ having fewer than four double points.  The best lower bound on this clasp number that can be obtained from previous signature based bounds is two.
\end{example}

\section{Comparison of bounds}\label{sec:compare}

Example~\ref{ex:split} illustrates a general procedure for finding a lower bound on the unknotting number.  For the moment we will call the outcome of that process $u_2(K)$.  We will not present a formal definition of this invariant. (There reader is invited to write down the details of the definition; it  requires defining the invariant that captures the minimum number of positive and negative crossing changes for each symmetric irreducible $\delta$ and then taking  the maximums  of each of these separately over all symmetric irreducible factors of $\Delta_K(x)$.  One must also consider the case   $\ups_\delta(K) < 0$, which we did not write down.)

 In this section we will compare $u_2(K)$ with the classical knot signature bound on $u(K)$;  we will temporarily denote the classical bound by $u_1(K)$.

Example~\ref{ex:split} presented  a knot  for which $u_2(K) = 2u_1(K)$.    By taking multiples of $K$ we can construct, for each $N>0$, a knot for which the classical signature bound on the unknotting number is $u(K) \ge 2N$, but for which our stronger invariants show that $u(K) \ge 4K$.  The next result states that this is the best possible. 

\begin{theorem} For all knots $K$, $u_1(K) \le u_2(K) \le 2u_1(K)$. \end{theorem}
\begin{proof}

Denote the minimum and maximum values of $\sigma_K(\omega)$ with $a$ and $A$.  Since the signature function takes on the value 0 near $\omega = 1$ ($t = 0$), we have $a \le 0 \le A$. By definition, $u_1(K) = (A - a)/2$.

For the convenience of the reader, we present the bounds on the signed number of crossing changes in Table~\ref{table}, covering the four possible cases.  For the moment, we let  $\mathcal{N}$ denote the minimum number of required changes from negative to positive and let $\mathcal{P}$ denote the minimum number of required crossing changes from positive to negative.  The table summarizes the result of Theorem~\ref{thm:signs}, including the cases in which $\ups<0$
\vskip.05in

\noindent  {\bf Part 1,  $\mathbf{u_2(K) \le 2 u_1(K)} $:}     Our bound on the unknotting number is the sum of entry from the ``$\mathcal{N}$'' column in the table, arising from some polynomial $\delta_1$ and an entry from the 
 ``$\mathcal{P}$'' column arising from a polynomial $\delta_2$, which might equal $\delta_1$. This sum will involve either one or two values of $\upj$, each divided by two.  Since each $\upj$ satisfies $0 \le \upj \le (A - a)/2$, the sum, after dividing by two, is less than or equal to  $(A - a)/2$.
 
 The sum of two entries also involves  terms of the form $(\ups - \downs)/2$, where  each term might arise from a different $\delta_i$.  (There are also cases in which either the $\ups$ or $\downs$ terms are replace with $0$.)  In any case, this sum is also bounded above by $(A-a)/2$.  
 
 Adding together these two sums yields a total that is less than or equal to $(A-a)$, which is $2u_1(K)$, as desired.
\begin{table}\label{table}

{ \renewcommand{\arraystretch}{1.5}
  \begin{tabular}{|r|r|r|}
 \hline 
  &$\mathcal{N} \ge $ &$\mathcal{P} \ge $\\  
   \hline
$ \ups \ge 0, \downs \le \upj$  &   $(\upj + \ups)/2$ & $ (\upj - \downs )/2$ \\    
  \hline    
 $ \ups \ge  0, \downs >  \upj$   &   $(\upj +\ups)/2$ &  $0$ \\   
  \hline
  $ \ups < 0, -\ups \le  \upj$  &  $(\upj + \ups)/2$ & $(\upj - \downs)/2$ \\   
  \hline
  $ \ups <  0, -\ups > \upj$  &   $0$ &  $(\upj - \downs)/2$      \\   
  \hline
 \end{tabular}
 }\vskip.1in
 \caption{Bounds on required signed crossing changes.}
\end{table}

 
\vskip.05in

 \noindent  {\bf Part 2,  $\mathbf{u_1(K) \le u_2(K)} $:}    We observe first that $A = \left|J_K(\omega)\right| + \sigma_K(\omega)$ for some value of $\omega$.  That value of $\omega$ is the root of an irreducible polynomial $\delta$. For that $\delta$, we can consider the possibilities that are listed in Table~\ref{table} and with care find that the bound on $\mathcal{N}$ must be of the form $(\upj + \ups)/2$.   Since $A \ge 0$, we must have $-\sigma_K(\omega) \le\left| J_K(\omega)\right|$.  Thus, the constraint that arises for  $\mathcal{N}$ must be at least  $A/2$.  
 
 In a similar manner, but working with  $-K$, we see that the constraint on the $\mathcal{P}$ must be at least $-a/2$.  Thus, the total must be at least $(A-a)/2$, as desired.

\end{proof}

\section{Gordian distance}  \label{sec:gordian}

The Gordian distance between knots $K$ and $J$, denoted $d_g(K,J)$, is defined to be the minimum number of crossing changes required to convert $K$ into $J$. Initial interest in $d_g$ arose in the classical knot theory setting, but as we will describe in Section~\ref{sec:4d}, it is  related to four-dimensional properties of knots and in particular is closely tied to a natural metric defined on the knot concordance group. References include~\cite{ baader,  baader2, blair-c-j-t-t, borodzik-friedl-powell,feller0,hirasawa,miyazawa,murakami}. 

\begin{theorem} $d_g(K,J) \ge u_2(K \cs -J)$.
\end{theorem}
\begin{proof} Since $u_2$ depends only on the signature function, it gives a lower bound on the number of crossing changes required to convert a given knot into a knot with trivial signature function.  If $K$ can be converted into $J$ with $u$ crossing changes, then $K\cs  -J$ can be converted into $J\cs -J$ with $u$ crossing changes.  The knot $J \cs -J$ is a slice knot, and thus has trivial signature.  (We will say more about such four-dimensional issues in Section~\ref{sec:4d}.)
\end{proof}

\begin{example}\label{ex:t3} As our only application, we consider the connected sum of torus knots $$K = T(3,10) \cs - T(2,15) \cs -T(5,6).$$  Its signature function is illustrated in Figure~\ref{fig:3ts}.  The Alexander polynomial factors as cyclotomic polynomials $\phi_6(x)^2\phi_{10}(x)^2\phi_{15}(x)^2\phi_{30}(x)^3$.  In the graphs, the points on the graph above these roots are marked, with the $\alpha$ points corresponding to roots of $\phi_{30}$;  similarly, $\beta, \gamma,$ and $\eta$, points correspond to roots of $\phi_{15}, \phi_{10},$ and $\phi_6$, respectively.

\begin{figure}[h]
\fig{.7}{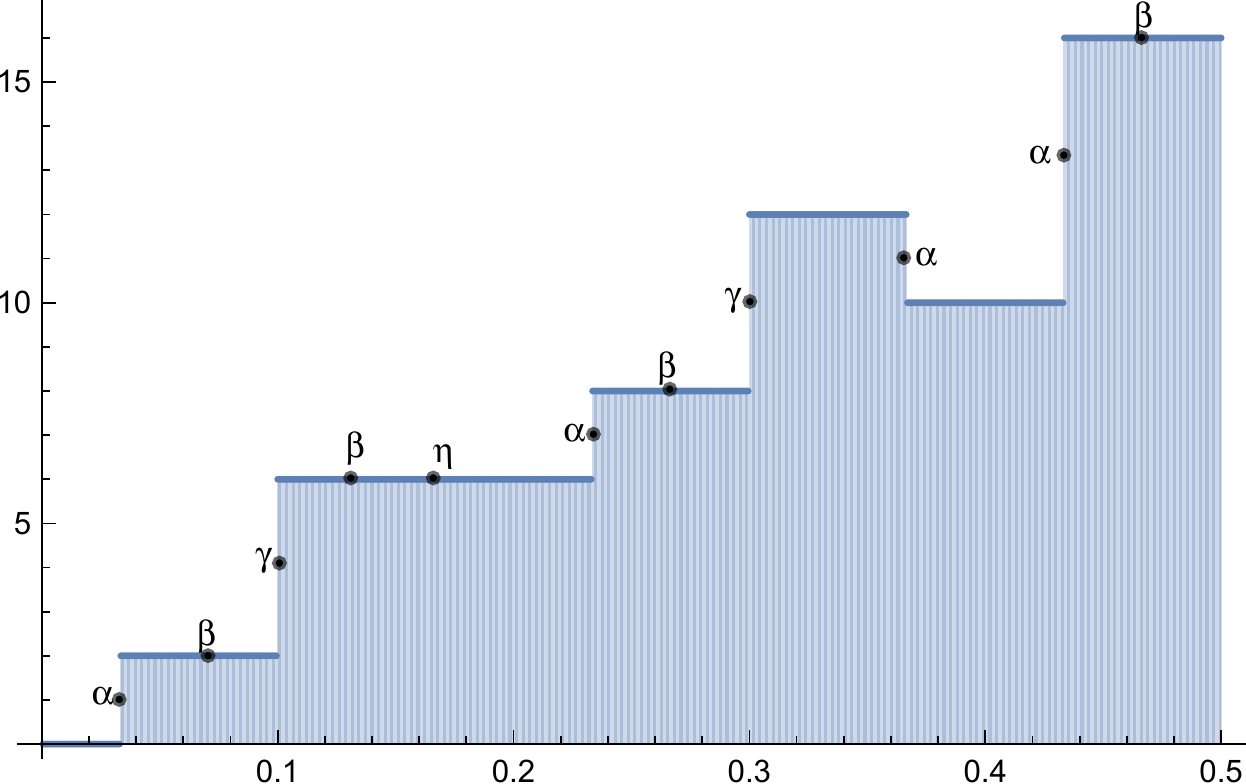}\label{fig:3ts}
 \caption{Signature function for $T(3,10) \cs -T(2,15) \cs -T(5,6)$.}\label{fig:twopolygraph}\end{figure}
 
 The classical signature bound on the unknotting number of $K$ is $16/2 = 8$.  Considering the polynomial $\phi_{30}$ we have the set of jumps is $\{1,1,-1,3\}$ and the set of signatures is $\{1,7,11,13\}$.  Thus  $(\upj, \downs, \ups) =( 3, 1 , 13)$.  Thus, we see that unknotting $K$ would require at least $(3+13)/2 = 8$ crossing changes from negative to positive, and at least $(3 -1)/2 =1$ crossing changes from positive to negative.  All together we have a total of at least 9 crossing changes required.
 
 In conclusion, we have the bound on the Gordian distance $d_g(T(3,10), T(2,15) \cs T(5,6)) \ge 9.$  
We will not compute Khovanov or Heegaard Floer invariants here, but note that the $\tau$, $\Upsilon$ and $s$--invariant bounds on the unknotting number and Gordian distances are all 8.

\end{example}

\section{Four-dimensional perspective}\label{sec:4d}
If  a knot $K$ can be unknotted with $k$ crossing changes, then it bounds an immersed disk in $B^4$ with $k$ transverse double points. The minimum number of  double points in an immersed disk bounded by $K$ in $B^4$, taken over all such immersed disks, has been called the {\it clasp number},  {\it the four-dimensional clasp number} or the {\it four-ball crossing number}.  For the moment, we denote this invariant $c(K)$. References include~\cite{kawamura, murakami-yasuhara, owens1}.  This invariant can be refined by considering the number of positive and negative double points in the immersed disk.  

In a similar way, if a sequence of crossing changes   converts a knot $K$ into a knot $J$, then there is an immersion of $S^1 \times I$ into $S^3 \times I$ (a {\it singular concordance}) with boundary $K\times \{0\} \cup J\times \{1\}$ such that the number of double points equals the number of crossing changes.  There is, in a way, a  converse.  From ~\cite[Proposition 2.1]{owens-strle} we have the following. 

\begin{theorem} If $K$ and $J$ bound a singular concordance with $p$ positive and $n$ negative double points, then there are knots $K'$ and $J'$, concordant to $K$ and $J$, respectively, such that $K'$ can be converted into $J'$ with $p$ positive and $n$ negative crossing changes.
\end{theorem}

For knots $K$ and $J$, we can define a distance $d_c(K,J)$ as the minimum number of double points in a singular concordance between the knots.  This induces a metric on the concordance group.   For $U$ the unknot, we have that the four-dimensional clasp number is equal to $d_c(K,U)$.

 Since concordant knots have the same signature functions, we have  the following result.

\begin{theorem}\label{thm:changes-singular}  For knots $K$ and $J$, $d_c(K,J) \ge u_2(K \cs -J)$.
\end{theorem}

\begin{example} The same computation  as done in Example~\ref{example1} shows that any singular concordance from $5_1$ to $   -10_{132}$ must have at least two positive double points and one negative double point.  In particular, the four-dimensional clasp number of $5_1 \cs 10_{132}$ is 3.
\end{example}

\begin{example} From Example~\ref{ex:t3} we have that $d_c(T(3,10), T(2,15) \cs T(5,6)) \ge 9$. 
\end{example}

\subsection{The four-genus}
The clasp number and four-genus of a knot are related by the following bound: $g_4(K) \le c(K)$.  This can be enhanced with the following observation: If $K$ bounds an immersed disk with $p$ positive and $n$ negative double points, then $g_4(K) \le \max(p,n)$.  Thus, lower bounds on $g_4$ provide lower bounds on the clasp number (and unknotting number).  

In the case that a knot is slice, $g_4(K) = 0$, the clasp number is also 0.  Multiples of the  square knot, $T(2,3) \cs -T(2,3)$, provide    examples of slice knots for which the unknotting number can be arbitrarily large.  In fact, using the homology of the $2$--fold branched cover~\cite{kinoshita, wendt}, one shows that $u\left(N( T(2,3) \cs -T(2,3))\right) = 2N$.  It is unknown whether there exists a knot $K$ with $g_4(K)  =1 $, but $c(K) > 2$;   in~\cite{livingston0, owens-strle} it is shown that there are knots with four-genus one that cannot be converted into slice knots using one crossing change.  Owens~\cite{owens1}  has identified  two-bridge knots $K_n$ with $g_4(K_n) = n$, $\sigma(K) = 2n$, and which cannot be converted into a slice knot with  $n$ crossing changes from negative to positive and any number of positive to negative crossing changes.  In particular, the knot $K_1 = - 7_4$ has $g_4(K_1) = 1$, $\sigma(K_1) = 2$, but any sequence of crossing changes that converts $K_1$ into a slice knot must include at least two crossing changes from negative to positive.  

Our main results, since they are providing {\it lower} bounds on $u(K)$ and $c(K)$, might not offer bounds on $g_4(K)$.  However, it is worth noting that the observations made in this paper offer a much simpler proof of this   result from~\cite{livingston1}. 

\begin{corollary}  Let $K \subset  S^3$ be a knot and let  $ \{\alpha_1, \ldots, \alpha_k\} \subset \ss$  be a nonempty subset of the complex roots of an irreducible rational polynomial $ \delta$.    Then $$g_4(K) \ge   \frac{1}{2}\Big( \max \{ \left|J_K(\alpha_i) \right|  \} +   \max \{ \left|\sigma_K(\alpha_i) \right|  \} \Big). $$
\end{corollary}

 \begin{proof} As stated in Theorem~\ref{thm:taylor2}, it follows from ~\cite{taylor}  that if $g_4(K) = g$, then $W_K \in W(\qq(x))$ has a  $2g \times 2g$ representative.  Consider the  diagonal form of such a representative, $W_K = [ f_1\delta, \ldots, f_m\delta, g_1, \ldots , g_n]$. It is clear that for all $\alpha_i$, $ \left|J_K(\alpha_i) \right| \le m $ and $    \max  \left|\sigma_K(\alpha_i) \right| \le n$.   It follows that $\max \{ \left|J_K(\alpha_i) \right|  \} +   \max \{ \left|\sigma_K(\alpha_i) \right|  \} \le m + n =2g$.
 \end{proof}


\section{The non-balanced signature function}\label{sec:non-balanced}\label{sec:non-balanced}

For a knot $K$ with Seifert matrix $V_K$, the signature of the matrix $(1-\omega)V_K + (1 - \overline{\omega})V_K^{\sf T}$ yields a well-defined function $s_K(\omega)$ on the unit circle.  The proof that $s_K(\omega)$ is a knot invariant is a consequence of  the fact that any two Seifert surfaces for a knot are stably equivalent~\cite{murasugi, trotter}.   The signature function we have been considering, $\sigma_K$, is defined by taking the two-sided average of $s_K$.  We have focused our attention on the balanced function because it defines a homomorphism on the knot concordance group.  There are example of slice knots for which $s_K$ is nontrivial~\cite{cha-livingston, levine2}, and thus it is not a knot concordance invariant. Here we briefly explore how  $s_K$ can be used to extract unknotting information that is not accessible via $\sigma_K$.  Of course, these results do not generalize to give concordance invariants.

Let $\delta(x)$ be an irreducible Alexander polynomial having   roots  $\{  \alpha_1, \ldots , \alpha_n\}$ on the unit circle.   Let $\Lambda_\delta = \qq[x,x^{-1}]_{(\delta)}$ denote the  ring formed by inverting all irreducible elements in $\qq[x,x^{-1}]$ other than $\delta(x)$.  The proof that   hermitian forms over $\qq(x)$ can be diagonalized is easily modified to the  case of hermitian forms over $\Lambda_\delta$.  One needs to check that the step-by-step diagonalization process can be adjusted so that division by $\delta(x)$ is not required.

Given this, the proof of  Lemma~\ref{cor:block} generalizes to the setting of $\Lambda_\delta$, and so the effect of crossing changes is determined by the signature functions for a pair of matrices of the following form:$$    
W_{\pm 1}(x) = \left(
\begin{array}{ c   c}
a(x) & \overline{b(x)} \\
{b(x)} &  d(x) +  \epsilon_{\pm} (1-x)(1-x^{-1})\\
\end{array}
\right).
$$
Here, all polynomials are in $\Lambda_\delta$, $\epsilon_- = 1$, and $\epsilon_+ = 0$.  

We now factor out powers of $\delta$ to rewrite this as 
$$    
\left(
\begin{array}{ c   c}
a'(x)\delta(x)^i & \overline{b'(x)}\delta(x)^j \\
{b'(x)\delta(x)^j} &  d(x) +  \epsilon_{\pm} (1-x)(1-x^{-1})\\
\end{array}
\right).
$$

Considering the difference of signatures, $\text{sign}(W_+(\alpha_i)) - \text{sign}(W_-(\alpha_i))$,  yields the following cases, all of which  are easily analyzed by considering diagonalizations. \vskip.05in
\begin{itemize}
\item  If $ i = 0$, the difference of signatures is determined by difference of values of    $$d'(\alpha_i) +  \epsilon_{\pm} (1-\alpha_i)(1-\alpha_i^{-1})$$ for some $d' \in \Lambda_\delta$.\vskip.05in

\item If $i \ne 0$ and $j =  0$, then both signatures are 0.\vskip.05in

\item If $i \ne 0$ and $j \ne 0$, then the difference of signatures is determined by difference of values of    $$d'(\alpha_i) +  \epsilon_{\pm} (1-\alpha_i)(1-\alpha_i^{-1})$$ for some $d' \in \Lambda_\delta$.\vskip.05in

\end{itemize}

The approach of our previous work now applies, and a simple consequence is the following.

\begin{lemma}  If $\delta(x) \in \zz[t,t^{-1}]$ has roots $\{\alpha_1, \ldots , \alpha_n\}$ on the unit circle and a crossing change from positive to negative is made to a knot $K$, then either all the values of $s_K(\alpha_i)$ increase by 1, or some increase by 2 and others are unchanged.
\end{lemma}  
 
 Rather than   exploit the dependance of this result on the choice of $\delta$, here  we will  present an  easily stated result, expressed in terms of the floor and ceiling function.  Recall that  $\max\{s_K(\omega)\} \ge 0$ and $\min \{s_K(\omega)\} \le 0$.   
 
 \begin{theorem}\label{thm:nonb} For a knot $K$, let $M = \max\{\sigma_K(\omega)\}$ and $m = \min\{\sigma_K(\omega)\}$.  The unknotting number satisfies 
 $$u(K) \ge \lceil  M/2 \rceil   - \lfloor m/2 \rfloor.$$
 
 \end{theorem}

\begin{example} 
The knot $8_{20}$ is slice, and hence its signature function $\sigma_K$ is identically 0.  However, its Alexander polynomial is $(t^2 - t+1)^2$, having a root at the sixth root of unity, $\xi_6$.  A direct computation shows that $s_K(\xi_6) = 1$, and this is the only nonzero value of the upper unit circle.   It is shown in~\cite{cha-livingston} that similar, but less explicit example abound.

Using such knots as $8_{20}$, one can   construct a knot $K$ for which there are two numbers on the upper half circle, $\omega_1$ and $\omega_2$, with the property that $s_K(\omega_1) = 3$, $s_K(\omega_2) = -3$, and $s_k(\omega) = 0$ for all other $\omega$ on the upper half circle.  According to Theorem~\ref{thm:nonb}, this knot has unknotting number at least 4.  Notice that this is one more than $\big(\max(s_K(\omega)) - \min(s_K(\omega))\big)/2$, as might be expected from classic bound based on $\sigma_K(\omega)$. 
\end{example}

 \subsection{Doubly slice knots}   A knot $K$ is called doubly slice if it is the cross-section of an unknotted two-sphere embedded in $S^4$.   The first proof of the existence of such knots appeared in~\cite{terasaka-hosakawa}.  Invariants of such knots have since been studied in much finer detail; see for instance,~\cite{kearton, t-kim, livingston-meier, meier, stoltzfus, sumners}.  The non-balanced signature function of a doubly slice knot is identically 0, and this provides a means of proving that some slice knots are not doubly slice.  
 
 The slicing number of a knot and the algebraic slicing number of a knot are the number of crossing changes required to convert a knot into a slice, respectively algebraically slice, knot.  One could similarly define a {\it double slicing number} of a knot.  Hence, we have:  
 \begin{theorem}For a knot $K$, let $M = \max\{\sigma_K(\omega)\}$ and $m = \min\{\sigma_K(\omega)\}$.  The number of crossing changes required to convert $K$ into a doubly slice knot is greater than or equal to
 $  \lceil  M/2 \rceil   - \lfloor m/2 \rfloor.$

 \end{theorem}

\end{document}